\documentclass{amsart}
\usepackage{amsmath,amsfonts,amssymb}
\usepackage{color}
\usepackage{hyperref}
\usepackage{enumerate}
\usepackage{ mathrsfs }
\usepackage{fullpage}
\usepackage{latexsym}
\usepackage{lineno,hyperref}
\modulolinenumbers[5]

    \newtheorem{theorem}{Theorem}[section]
    
    \newtheorem{proposition}[theorem]{Proposition}
    \newtheorem{corollary}[theorem]{Corollary}

    \newenvironment{remark}[1][Remark]{\begin{trivlist}
    \item[\hskip \labelsep {\bfseries #1}]}{\end{trivlist}}

\newcommand{\inc}{\in}
\newcommand{\PSO}{\mathrm{PSO}^-}
\newcommand{\PO}{\mathrm{P} \Omega ^-}
\begin{document}

%\begin{frontmatter}

\title{A note on relative hemisystems of Hermitian generalised quadrangles}

\author[uwa]{John Bamberg}
\address[J. Bamberg and M. Lee]{
Centre for the Mathematics of Symmetry and Computation\\
School of Mathematics and Statistics\\
The University of Western Australia\\
35 Stirling Highway, Crawley, WA 6009, Australia.}
\email[J. Bamberg]{john.bamberg@uwa.edu.au}
\author[uwa]{Melissa Lee}
\email[M. Lee]{20945795@student.uwa.edu.au}

\author[wm]{Eric Swartz}

\address[E.Swartz]{
Department of Mathematics\\
College of William and Mary\\
P.O. Box 8795
Williamsburg, VA 23187-8795, USA.}
\email[E. Swartz]{easwartz@wm.edu}
%\begin{keyword}
%Hermitian variety, classical polar space, relative hemisystem
%\MSC[2010]{11E39, 51A50, 51E12}
%\end{keyword}
\maketitle

\begin{abstract}
In this paper we introduce a set of sufficient criteria for the construction of relative hemisystems of the Hermitian space $\mathrm{H}(3,q^2)$, unifying all known infinite families. We use these conditions to provide new proofs of the existence of the known infinite families of relative hemisystems. Reproving these results has allowed us to find new relative hemisystems closely related to an infinite family of Cossidente's, and develop techniques that are likely to be useful in finding relative hemisystems in future.                                                                                                                                                                                                                                                                                                                                                                                                                                                                                                                                                                                                                                                                                                                                                                                                                                                                                                                                                                                                                                                                                                                                                                                                                                                                                                                                                                                                                                                 
\end{abstract}

%\begin{keyword}
%\texttt{elsarticle.cls}\sep \LaTeX\sep Elsevier \sep template
%\MSC[2010] 00-01\sep  99-00
%\end{keyword}

%\end{frontmatter}

\section{Introduction}

Hemisystems in finite geometry have a short but eventful history, dating back to B. Segre's definition of them as a special case of a regular system in 1965 \cite{MR0213949}. 
A hemisystem on the Hermitian space $\mathrm{H}(3,q^2)$, $q$ odd, is a set of lines $\mathcal{L}$ on $\mathrm{H}(3, q^2)$ such that every point in $\mathrm{H}(3, q^2)$ has half of the lines incident with it in $\mathcal{L}$. Hemisystems are of great interest 
because they give rise to partial quadrangles, strongly regular graphs and cometric $Q$-antipodal association schemes \cite[\textsection 7.5.1]{MR3092674}. In his 1965 treatise, Segre gave an example of a hemisystem on $\mathrm{H}(3, 3^2)$ and proved that
it was unique (up to equivalence) on this Hermitian space \cite{MR0213949}.
For forty years after their introduction, no new examples of hemisystems were found, and Thas conjectured that none existed on $\mathrm{H}(3,q^2)$ for $q >3$ \cite[\textsection 9.5]{buekenhout1995handbook}. %REFERENCE RIGHT?
This conjecture was proven false when Penttila and Cossidente discovered a new infinite family of hemisystems in 2005 \cite{cossidente2005hemisystems}. 

In 2011, Penttila and Williford introduced the notion of 
\textit{relative hemisystems}, an analogous concept to hemisystems that exist on $\mathrm{H}(3,q^2)$, for $q$ even \cite{penttila2011new}. They were motivated by the desire to construct an example of an infinite family of primitive cometric association schemes that do not arise from distance regular graphs. 
Prior to their paper, only sporadic examples of such association schemes were known \cite[\textsection 1]{penttila2011new}. 

Let $Q$ be a generalised quadrangle of order $(q^2, q)$, containing a generalised quadrangle $Q'$ of order $(q,q)$, where $q$ is a power of two. 
Each of the lines in $Q$ meet $Q'$ in either $q+1$ points or are disjoint from it. A subset $\mathcal{R}$ of the lines in $Q \setminus Q'$ is a \textit{relative hemisystem} of $Q$ with respect to $Q'$ if
for every point $P$ in $Q \setminus Q'$, exactly half the lines through $P$ disjoint from $Q'$ lie in $\mathcal{R}$. Notice that this definition is well defined, because the number of lines through $P$ disjoint from $Q'$ is $q$.

Penttila and Williford concluded their paper with an open question on the existence of non-isomorphic relative hemisystems
on $\mathrm{H}(3,q^2)$, $q$ even. Cossidente resolved this question two years later by constructing an infinite family of non-isomorphic relative hemisystems,
each admitting $\mathrm{PSL}(2,q)$ as an automorphism group \cite{MR3081646}, and another admitting a group of order $q^2(q+1)$, for each $q\geqslant 8$, a power of two \cite{cossidente2013new}. In 2014, Cossidente and Pavese constructed a relative hemisystem arising from a Suzuki-Tits ovoid on $\mathrm{H}(3,64)$ \cite{MR3252665}. They conjectured that this relative hemisystem is sporadic.

Cossidente \cite{MR3081646} shows that his construction generates a new infinite family by showing that he has created more relative hemisystems than the number generated by the Penttila-Williford construction. Through finding all of the relative hemisystems invariant under $\mathrm{PSL}(2,q)$ for varying values of $q$, we have shown that there are actually several more inequivalent infinite families of relative hemisystems that arise from this construction (see Remark \ref{remark:Cossidente}). For $q = 16$, we found by computer that there are five inequivalent relative hemisystems that each admit $\mathrm{PSL}(2, q)$, and we conjecture that the number of inequivalent infinite families increases with $q$.

In this paper, we state a set of new sufficient criteria for relative hemisystems, which unifies all currently known infinite families of relative hemisystems.
We begin with some background, and briefly recall the constructions of the three known infinite families. We go on to prove a series of results, culminating in 
Theorem \ref{majortheorem} which states sufficient conditions to determine a relative hemisystem. Using these, we provide new explicit proofs of the Penttila-Williford and Cossidente results.
Finally, we briefly discuss a computational result classifying all of the relative hemisystems on $\mathrm{H}(3,4^2)$, and some results partially classifying the relative hemisystems on $\mathrm{H}(3,8^2)$.

\section{Background information}
A finite generalised quadrangle of order $(s,t)$ is an incidence structure of points and lines such that:
\begin{itemize}
\item Any two points are incident with at most one line.
\item Every point is incident with $t+1$ lines.
\item Every line is incident with $s+1$ points.
\item For any point $P$ and line $\ell$ that are not incident, there is a unique point $P'$ on $\ell$ that is collinear with $P$.
\end{itemize}

If we take the point-line dual of a generalised quadrangle of order $(s,t)$, we obtain another generalised quadrangle, of order $(t,s)$.

There are many examples of generalised quadrangles; for a good reference for the finite cases, see \cite{payne2009finite}. 
For the purposes of this paper, we are interested in two families of generalised quadrangles.
\begin{itemize}
 \item \textbf{$\mathrm{H}(3,q^2)$}, the set of all totally isotropic lines and points with respect to a Hermitian form on the projective space $\mathrm{PG}(3,q^2)$. These are generalised quadrangles of order $(q^2, q)$, for $q$ a prime power.
 \item \textbf{$\mathrm{W}(3,q)$}, the set of all totally isotropic points and lines of a symplectic form on $\mathrm{PG}(3,q)$. These are generalised quadrangles of order $(q,q)$, for $q$ a prime power.
\end{itemize}
In this paper we will only work with fields with even characteristic, even though some results may hold for odd prime powers. For this reason, from now on we will assume that $q = 2^k$ for some positive integer $k$.

For a given $q$, we can find an embedding of $\mathrm{W}(3,q)$ in $\mathrm{H}(3,q^2)$. For a detailed description of the construction of this embedding, the reader is directed to \cite[Section 4.5]{kleidman1990subgroup} and our construction in Section \ref{newproofs}. 
This gives us the setup needed for the construction of a relative hemisystem. When constructing a relative hemisystem,
we are only concerned with the set of points of $\mathrm{H}(3,q^2)$ which are outside of $\mathrm{W}(3,q)$ and the lines of $\mathrm{H}(3,q^2)$ disjoint from $\mathrm{W}(3,q)$. 
For conciseness, we call the former the set of \textit{external points},
denoted $\mathcal{P}_E$, and the latter the set of \textit{external lines}, denoted $\mathcal{L}_E$.

Throughout the paper, we will also make reference to the collineation group of the Hermitian space $\mathrm{H}(3,q^2)$, which shall be denoted as the projective unitary group $\mathrm{P} \Gamma \mathrm{U}(4,q)$.
We use $\Omega / G$ to denote the set of orbits of a set $\Omega$ under a group $G$. Also recall that if $N$ is a normal subgroup of $G$, then $G$ acts on $\Omega / N$ in its action on sets.
We will say that a group $G$ acts \textit{semiregularly} in its action on a set $\Omega$ if the only elements in $G$ that fix an element of $\Omega$ are those in the kernel of the action. Note that definition of semiregular may be different to definitions in other areas of the literature. 

Quadrics are also integral to many of the results in this paper. Quadrics in a three dimensional projective space are the totally singular points and 
lines (if they exist) of a quadratic form. In $\mathrm{PG}(3,q)$, there are two sorts of non-singular quadrics -- hyperbolic and elliptic.
Hyperbolic quadrics, denoted $\mathrm{Q}^+(3,q)$ are quadrics which contain totally singular lines. Otherwise, the quadric is elliptic, denoted $\mathrm{Q}^-(3,q)$. For further background on quadrics, see \cite[\S22]{MR1363259}.

Suppose $q$ is even and $\mathcal{Q}^+$ is an irreducible hyperbolic quadric of $\mathrm{PG}(3,q^2)$ that shares a tangent plane with $\mathrm{H}(3,q^2)$ at a common point. Then the intersection of $\mathcal{Q}^+$ and $\mathrm{H}(3,q^2)$ has size $q^2+1$ and is an elliptic quadric $\mathrm{Q}^-(3,q)$ \cite{quadherm}. 
\section{The known families of relative hemisystems}
Here, we give brief descriptions of the constructions of the known infinite families of relative hemisystems.
\subsection{The Penttila-Williford relative hemisystems}
\label{PWRHs}
The first example of an infinite family of relative hemisystems, admitting $\mathrm{P} \Omega ^- (4,q)$ as an automorphism group, was given by Penttila and Williford in their paper introducing the concept \cite{penttila2011new}.
Their construction \cite[Theorem 5]{penttila2011new} considers the action of the normaliser of a Singer cycle of $\mathrm{P} \Omega^-(4,q)$ on $\mathrm{H}(3,q^2)$. They use it to prove that 
$\mathrm{P} \Omega^-(4,q)$, 
$q$ even, $q \geqslant 2$ has two orbits on external lines. It transpires that these orbits form two relative hemisystems, $H_1$ and $H_2$. They further show that there exists an involution $t$, which fixes the points of $\mathrm{W}(3,q)$ and switches $H_1$ and $H_2$ \cite{penttila2011new}.
\subsection{The Cossidente relative hemisystems}
\label{CRHs}
Apart from the Penttila-Williford family of relative hemisystems, the only other known infinite families of relative hemisystems are the two discovered by Cossidente \cite{MR3081646, cossidente2013new} in 2013.
 Both of these families are \textit{perturbations} of the Penttila-Williford relative hemisystems.

 The first family on $\mathrm{H}(3,q^2)$, $q$ even and $q>4$, admits the linear group $\mathrm{PSL}(2,q)$ as an automorphism group \cite{MR3081646}. Cossidente constructed it by taking the two Penttila-Williford
 relative hemisystems $H_1$ and $H_2$ and considering the stabiliser of a conic section of the elliptic quadric $\mathrm{Q}^-(3,q)$ in $\mathrm{W}(3,q)$ fixed by $\mathrm{P} \Omega^-(4,q)$. 
 This stabiliser is isomorphic to $\mathrm{PSL}(2,q)$ and 
 does not act transitively on $H_1$ and $H_2$. Cossidente then uses the involution $t$, switching $H_1$ and $H_2$ from the Penttila-Williford proof \cite[Theorem 5]{penttila2011new} 
 to delete some orbits of $H_1$ under $\mathrm{PSL}(2,q)$ and replace them with their images under $t$. Since the number of ways this can be done 
 outnumbers the number of Penttila-Williford relative hemisystems, Cossidente has constructed a new infinite family.

 The second infinite family of relative hemisystems discovered by Cossidente admits a group of order $q^2(q+1)$ for each $q$ even and $q>4$. 
 The construction of this infinite family is very similar to the last. Choose a point $P$
 of an elliptic quadric $\mathrm{Q}^-(3,q)$, which is an ovoid of $\mathrm{W}(3,q)$. Let $M$ be the subgroup of the stabiliser of $P$ in $\mathrm{P} \Omega^-(4,q)$ of order $q^2(q+1)$.
 Instead of orbits under $\mathrm{PSL}(2,q)$, Cossidente considers orbits of $H_1$ and $H_2$ under $M$, deleting orbits of $H_1$ and replacing them by their image under the involution $t$. 
 Since the number of relative hemisystems invariant under $M$ exceeds that of the Penttila-Williford relative hemisystems, Cossidente must have found another infinite family of relative hemisystems.
 
 %Notice that in both cases, since there are deletions and additions of orbits, the automorphism groups admitted by these families are not contained in $\PO(4,q)$. 
 
\section{New sufficient conditions for relative hemisystems}
\label{magic}
Let $\mathcal{Q}^+$ be a hyperbolic quadric which intersects a Hermitian space $\mathrm{H}(3,q^2)$ in an elliptic quadric isomorphic to $\mathrm{Q}^-(3,q)$. 
The stabiliser of $\mathcal{Q}^+$ in $\mathrm{P} \Gamma \mathrm{U} (4,q)$ is isomorphic to the orthogonal group $\mathrm{PSO}^-(4,q)$ \cite{penttila2011new}; 
and the subgroup of $\mathrm{P} \Gamma \mathrm{U}(4,q)_{\mathcal{Q}^+}$ that stabilises the two families of reguli of $\mathcal{Q}^+$ is 
isomorphic to $\mathrm{P} \Omega^-(4,q)$ \cite{penttila2011new}. Penttila and Williford prove that $\mathrm{P}\Omega ^-(4,q)$ has two orbits on external lines, and $\mathrm{PSO}^-(4,q)$ is transitive on
external lines.

Taking $\ell \in \mathcal{L}_E$, the size of the orbit of $\ell$ under $\mathrm{PSO}^-(4,q)$ is twice as large as the orbit of $\ell$ under $\mathrm{P}\Omega ^-(4,q)$. 
By the Orbit-Stabiliser Theorem, 
$|\mathrm{PSO}^-(4,q) : \mathrm{PSO}^-(4,q)_{\ell}| = 2|\mathrm{P}\Omega ^-(4,q): \mathrm{P}\Omega ^-(4,q)_{\ell}|$. 
Since $|\mathrm{PSO}^-(4,q):\mathrm{P} \Omega^-(4,q)| = 2$ \cite[Table 2.1d]{kleidman1990subgroup}, it follows that $\mathrm{PSO}^-(4,q)_{\ell} = \mathrm{P} \Gamma \mathrm{U}(4,q)_{\mathcal{Q}^+, \ell} = \mathrm{P}\Omega ^-(4,q)_{\ell}$. 
Therefore, the size of the orbit of $\ell^{\mathrm{P}\Omega ^-(4,q)}$ under $\mathrm{PSO}^-(4,q)$ is 
$$ |\mathrm{PSO}^-(4,q) : \mathrm{PSO}^-(4,q)_{\ell}\mathrm{P}\Omega^-(4,q)| = |\mathrm{PSO}^-(4,q) : \mathrm{P}\Omega^-(4,q)|  =2$$ and hence $\mathrm{PSO}^-(4,q) = \mathrm{P} \Gamma \mathrm{U}(4,q)_{\mathcal{Q}^+}$ 
acts semiregularly on the orbits of $ \mathrm{P}\Omega ^-(4,q)$ on external lines.

\begin{proposition}
\label{cond2prop}
  Let $G$ be a subgroup of $\mathrm{P} \Gamma \mathrm{U}(4,q)_{\mathcal{Q}^+}$, where $\mathcal{Q}^+$ is a hyperbolic quadric meeting $\mathrm{H}(3,q^2)$ in an elliptic quadric. 
  Let $D$ be the subgroup of $\mathrm{P} \Gamma \mathrm{U}(4,q)_{\mathcal{Q}^+}$ that stabilises the two families of reguli of $\mathcal{Q}^+$. If $G$ is not contained in $D$, then $G$ acts semiregularly on the orbits of $G\cap D$ on external lines.
\end{proposition}

\begin{proof}
Suppose the contrary, that $G$ fixes an orbit $\ell^{G\cap D}$. Then $\ell^G = \ell^{G\cap D}$ and by the Orbit-Stabiliser Theorem,
 $$|G\cap D: (G\cap D)_\ell | =| \ell^{G\cap D} | = | \ell^G | = | G : G_\ell |$$
and hence 
\begin{equation}
\label{prop1eqn}
 | G : G\cap D | = | G_\ell : (G\cap D)_\ell |.
\end{equation}

Now from the discussion at the beginning of this section, we have $\mathrm{P} \Gamma \mathrm{U}(4,q)_{\mathcal{Q}^+,\ell} = D_\ell$. 
So, $G$ is a subgroup of $\mathrm{P} \Gamma \mathrm{U}(4,q)_{\mathcal{Q}^+}$, and $G_{\ell}$ is a subgroup of $D_{\ell}$. Hence $G_\ell = (G \cap D)_\ell$, and so by Equation \ref{prop1eqn}, $|G:G \cap D| = 1$. 
This implies that $G = G \cap D$ and $G$ is a subgroup of $D$. This is a contradiction by the definitions of $D$ and $G$. Therefore, $G$ must act semiregularly on the orbits of $G\cap D$ on external lines.

\end{proof}

 \begin{theorem}
\label{majortheorem}
Suppose $G < \overline{G}$, where $\overline{G}$ is a subgroup of the collineation group of $\mathrm{H}(3,q^2)$ stabilising $\mathrm{W}(3,q)$. Further suppose that $\overline{G}$ and $G$ satisfy the following conditions: 
 \begin{enumerate}
\item $| \overline{G}{:} G | = 2$,
\item $\overline{G}$ acts semiregularly on $\mathcal{L}_E/G$,
\item $\mathcal{P}_E / \overline{G} = \mathcal{P}_E / G$.
\end{enumerate}

Write $\mathcal{L}_E / \overline{G} = \{ \ell_1^{\overline{G}}, \ell_2^{\overline{G}}, \dots \ell_n^{\overline{G}} \} $, where each $\ell_i$ is a representative from a distinct orbit of $\overline{G}$ on $\mathcal{L}_E$. Then $\bigcup_{i=1}^n \ell_i^G$ is a relative hemisystem \index{relative hemisystem}.  
\end{theorem}

\begin{proof}
 Let $X$ be an external point. For $\ell \in \mathcal{L}_E$, we define the \textit{line orbit incidence number} \index{line orbit incidence number} as
 $$n_{X, \ell}^G = | \{ m \in \ell^G \mid X\inc m \} |$$
First note that for all $g \in G$, \begin{equation}\label{loin} n_{X, \ell}^G = n_{X^g, \ell}^G \end{equation}  because $X \inc m$ if and only if $ X^g \inc  m^g$.
  Now, since $ |\overline{G}{:} G|=2$, for all $\ell \in \mathcal{L}_E$ there exist $m_1, m_2 \in \mathcal{L}_E$ such that $\ell^{\overline{G}} = m_1^G \cup m_2^G$. Then, we can find $t \in \overline{G}$ such that
 $(m_1^G)^t = m_2^G$. Since $\overline{G}$ acts semiregularly on $\mathcal{L}_E/G$, the orbits of $\overline{G}$ on $\mathcal{L}_E/G$ have size two and 
 $G$ is the kernel of the action.
 Therefore, there exists an element $t \in \overline{G}$ such that for all $\ell \in \mathcal{L}_E$, we have $ \ell^{\overline{G}} = \ell^G \dot{\cup} (\ell^t)^G$.
 We then have $n_{X, \ell}^{\overline{G}} = n_{X, \ell^t}^G + n_{X, \ell}^G$. 
 Consider $n_{X, \ell^t}^G = | \{ m \in (\ell^t)^G \mid X\inc m \}|$. Since $G$ has index two and is therefore a normal subgroup of $\overline{G}$, $n_{X, \ell^t}^G = | \{ n^t\in (\ell^G)^t \mid X\inc n^t \}| = 
 | \{ n\in (\ell^G) \mid X^{t^{-1}}\inc n \}|$. Now, since $\overline{G}$ and $G$ have the same orbits on external points by Condition 3, , there exists $u \in G$ such that $X^{t^{-1}} = X^u$. So $n_{X, \ell^t}^G = | \{ n\in (\ell^G) \mid X^u\inc n \}|$ = $n_{X^u, \ell}^G = n_{X,\ell}^G$, by \eqref{loin}.
 Therefore,  
 \begin{equation}
 \label{majorhalf}
  n_{X, \ell}^{\overline{G}} = 2 n_{X, \ell}^G
 \end{equation}

 Consider the orbit representatives $\ell_1 , \ell_2, \dots \ell_n$ of $\mathcal{L}_E/ \overline{G}$.
 The number $q$ of external lines incident with $X$ is then equal to the sum of the line orbit incidence numbers $n_{X,\ell_i}^{\overline{G}}$, for $i\in \{1, \dots n \}$.
  Then, from Equation \ref{majorhalf}, $q/2 = \sum_{i=1}^n n_{X, \ell_i}^G$. So the number of lines of $\bigcup_{i=1}^n \ell_i^G$ incident with $X$ is $q/2$. 
  Therefore, since $X$ was an arbitrary external point, $\bigcup_{i=1}^n \ell_i^G$ is a relative hemisystem.

\end{proof}

We remark that Theorem \ref{majortheorem} is similar to the technique developed by Bayens \cite[\textsection 4.4]{MR3167190} to construct hemisystems of higher dimensional Hermitian spaces.
When we are dealing with the Penttila-Williford \index{relative hemisystem}relative hemisystems and perturbations of them, we can condense the criteria given in Theorem \ref{majortheorem} to two sufficient criteria to determine a relative hemisystem.
We state these conditions in the following corollary to Theorem \ref{majortheorem}.

\begin{corollary}
\label{Cosscoroll}
Suppose $\overline{G}$ is a subgroup of $\mathrm{PSO}^-(4,q)$ and $G$ is the intersection of $\overline{G}$ and $\PO (4,q)$.
Further suppose that $\overline{G}_P$ is not contained in $\PO(4,q)$ for all external points $P \in \mathcal{P}_E$.
Then $(G,\overline{G})$ satisfies the conditions given in Theorem \ref{majortheorem} and thus determines a relative hemisystem.\index{relative hemisystem}
\end{corollary}

\begin{proof}%\hspace{-0.1cm}$^*$
First notice that if  $\overline{G}_P$ is not contained in $\PO(4,q)$ for all external points $P \in \mathcal{P}_E$, then there exists an element $g \in \overline{G}$ such that $g \notin \PO(4,q)$. So $\overline{G}$ is not contained in $\PO (4,q)$.
We have $|\overline{G}{:} G| = |\overline{G} {:} \overline{G} \cap \PO (4,q)|= |\overline{G} \cdot \PO (4,q){:}\PO (4,q)|= |\PSO(4,q){:}\PO (4,q)| = 2$. 
Let $\ell \in \mathcal{L}_E$.   Now, the discussion at the beginning of Section \ref{magic} implies that for any $\ell \in \mathcal{L}_E$, $\PSO(4,q)_{\ell} = \PO(4,q)_{\ell}$. 
Thus $\overline{G}_\ell=\PSO (4,q)_\ell$ is contained in $\PO(4,q)$. Therefore, $G_\ell = \overline{G}_\ell \cap \PO (4,q)=\overline{G}_\ell$. 
Now, since $\overline{G}_P$ is not contained in $\PO(4,q)$, $\overline{G}_P \neq \overline{G}_P \cap \PO (4,q)$. So the stabiliser of $P$ under $\overline{G}$ is not equal to the stabiliser under $G$. By the Orbit-Stabiliser Theorem, 
$$\frac{|P^{\overline{G}}|}{|P^G|} = \frac{|\overline{G}{:}\overline{G}_P|}{|G{:}G_P|} = \frac{|\overline{G}{:}G|}{|\overline{G}_P{:}G_P|}.$$
Since $|\overline{G}{:}G| = 2$, we have $|\overline{G}_P{:}G_P| = 2$ and hence $P^{\overline{G}} = P^G.$
Therefore, $(G,\overline{G})$ satisfies the conditions of Theorem \ref{majortheorem}.

\end{proof}

\section{New proofs of the known infinite families}
\label{newproofs}
Let $\psi(x_1,x_2) = x_1^2 + \upsilon^{q+1} x_2^2  + x_1x_2$ be a form with $\upsilon \in \mathrm{GF}(q^2)$ satisfying $\upsilon^q + \upsilon = 1$. Then, $\psi$ is irreducible over $\mathrm{GF}(q)$.
Next, notice that the totally singular points and lines of the form
\begin{equation}
 \label{gammaquad}
 \mathcal{Q}^+{:} \,  x_1^2 + \upsilon^{q+1} x_2^2 +x_1x_2 + x_3x_4
\end{equation}
define a hyperbolic quadric that intersects the Hermitian space defined by the form $x_1x_2^{q} + x_2x_1^q +x_3x_4^q+x_4x_3^q$ over $\mathrm{GF}(q^2)$
in an elliptic quadric. This elliptic quadric's defining equation is simply the equation for $\mathcal{Q}^+$ restricted to $\mathrm{GF}(q)$.
In this section, we will represent lines in array form, as the span of two projective points. 
The reguli of the hyperbolic quadric are:
 \begin{align*}
  \mathscr{R}_1 & = \lbrace \left[ \begin{smallmatrix} \upsilon^q&1 &0&\lambda \\ \lambda \upsilon&\lambda &1&0 \\ \end{smallmatrix} \right]\mid \lambda  \in \mathrm{GF}(q^2) \rbrace \cup\lbrace \left[ \begin{smallmatrix} 0&0&0&1 \\ \upsilon&1&0&0\\ \end{smallmatrix} \right] \rbrace,\\
&\\
\mathscr{R}_2 & = \lbrace \left[ \begin{smallmatrix} \upsilon^q&1&\lambda&0 \\ \lambda \upsilon&\lambda &0 &1 \\ \end{smallmatrix} \right] \mid \lambda  \in \mathrm{GF}(q^2) \rbrace \cup\lbrace \left[ \begin{smallmatrix} 0&0&1&0 \\ \upsilon&1&0&0\\ \end{smallmatrix} \right]\rbrace.
\end{align*}
\begin{proposition}
 The Penttila-Williford family of relative hemisystems, admitting $\mathrm{P}\Omega^-(4,q)$ as an automorphism group for each $q$ even, satisfies Corollary \ref{Cosscoroll}, with associated groups $G = \PO(4,q)$ and $\overline{G} = \PSO(4,q)$, which fix $\mathcal{Q}^+$.
\end{proposition}
\begin{proof}%\hspace{-0.1cm}$^*$
\label{pwproof}
 Let $\mathrm{H}(3,q^2)$ be the Hermitian space defined by the form $x_1x_2^{q} + x_2x_1^q +x_3x_4^q+x_4x_3^q$ over $\mathrm{GF}(q^2)$. The embedded 
 symplectic space $\mathrm{W}(3,q)$ is the restriction of the Hermitian form to $\mathrm{GF}(q)$. Recall that $\overline{G} = \PSO(4,q)$ is isomorphic to the stabiliser of $\mathcal{Q}^+$, and $G = \PO(4,q)$ is isomorphic to
 the stabiliser in $\overline{G}$ of the reguli of $\mathcal{Q}^+$.
 Consider $g \in  \overline{G}$ defined by 
 \[
  g = \left( \begin{smallmatrix} 1&0&0&0 \\ 0&1&0&0 \\ 0&0&0&1 \\ 0&0&1&0 \\ \end{smallmatrix} \right).
 \]
We claim that $g$ does not fix the reguli of the hyperbolic quadric $\mathcal{Q}^+$.
Finding the image of $\mathscr{R}_1$ under $g$ gives us
\[
 \left[ \begin{smallmatrix} \upsilon^q&1 &0&\lambda \\ \lambda \upsilon&\lambda&1&0 \\ \end{smallmatrix} \right]^g =   \left[ \begin{smallmatrix} \upsilon^q&1&\lambda&0 \\ \lambda \upsilon&\lambda &0 &1 \\ \end{smallmatrix} \right] 
\]
for $\lambda \in \mathrm{GF}(q^2)$, and 
\[
 \left[ \begin{smallmatrix} 0&0&0&1 \\ \upsilon&1&0&0\\ \end{smallmatrix} \right]^g  =  \left[ \begin{smallmatrix} 0&0&1&0 \\ \upsilon&1&0&0\\ \end{smallmatrix} \right]
\]
which are exactly the lines of $\mathscr{R}_2$. Since $g$ has order two, $g^{-1} = g$ and so $\mathscr{R}_2$ must map to $\mathscr{R}_1$ under the action of $g$.
Therefore, since $G$ stabilises the reguli of the hyperbolic quadric from the beginning of this section, $g \in \overline{G} \setminus G$.
Now, notice that $P_\omega = (\omega, 0, 1,1) \in \mathrm{H}(3,q^2)$ for all $\omega \in \mathrm{GF}(q^2)$, and if we take $\omega \in \mathrm{GF}(q^2) \setminus \mathrm{GF}(q)$, then $P_\omega$ is an external point. Then $P_\omega^g = (\omega,0,1,1)^g = (\omega,0,1,1)$ and therefore $g$ fixes $P_\omega$.
So $g \in  \overline{G}_{P_\omega}$, but $g \notin G_{P_\omega}$ because $g \notin G$. Therefore, $\overline{G}_{P_\omega} \neq G_{P_\omega} = \overline{G}_{P_\omega} \cap G$ and $\overline{G}_{P_\omega}$ is not contained
in $G$. Finally, from \cite{penttila2011new}, $G=\mathrm{P} \Omega^-(4,q)$ is transitive on external points. It immediately follows that $\overline{G}=\mathrm{PSO}^-(4,q)$ is transitive on external points as well.
This implies that $\overline{G}_Q$ is not contained in $G$ for all external points $Q \in \mathcal{P}_E$.
Therefore, $(G, \overline{G})$ determine a relative hemisystem for every $q$ even by Corollary \ref{Cosscoroll}, and this relative hemisystem belongs to the Penttila-Williford family of relative hemisystems.

\end{proof}
 We now prove that the Cossidente relative hemisystems satisfy the condition in Corollary \ref{Cosscoroll}. We begin by defining collineations $\tau$ and $\phi$ as follows:
\begin{align*}
 \tau{:} (x_1,x_2,x_3,x_4) & \mapsto (x_1+ x_2,x_2,x_3,x_4),\\
 \phi{:} (x_1,x_2,x_3,x_4) & \mapsto (x_1^q,x_2^q,x_3^q,x_4^q).
\end{align*}
\begin{theorem}
 The first family of Cossidente relative hemisystems\index{relative hemisystem} admitting $\mathrm{PSL}(2,q)$ as a setwise stabiliser (described in \cite{MR3081646}) satisfies Corollary \ref{Cosscoroll}. The associated groups are $G = \mathrm{PSL}(2,q) \times \langle \tau\phi \rangle$ and $\overline{G} = 
\mathrm{PSL}(2,q) \times \langle \tau, \phi \rangle$.
\end{theorem}
\begin{proof}%\hspace{-0.1cm}$^*$
Let $\mathrm{H}(3,q^2)$ be the Hermitian space in $\mathrm{PG}(3,q^2)$ with defining Gram matrix 
\[
 H = \left( \begin{smallmatrix} 0 & 1 & 0 & 0 \\
1 & 0 & 0 & 0 \\
0 & 0 & 0 & 1\\
0 & 0 & 1 & 0 \end{smallmatrix} \right).
\]
Note that this is the Gram matrix for the Hermitian space defined at the beginning of this section.
 Let $\mathcal{Q}^+$ be the hyperbolic quadric described in Equation \ref{gammaquad}. 
 The Baer subspace that contains the symplectic space $\mathrm{W}(3,q)$ and the elliptic quadric $\mathcal{Q}^- = \mathcal{Q}^+ \cap \mathrm{H}(3,q^2)$ consists of points whose coordinates lie solely in $\mathrm{GF}(q)$.
% Recall the definitions of collineations $\tau$ and $\phi$:
% \begin{align*}
%  \tau{:} (x_1,x_2,x_3,x_4) & \mapsto (x_1+ x_2,x_2,x_3,x_4),\\
%  \phi{:} (x_1,x_2,x_3,x_4) & \mapsto (x_1^q,x_2^q,x_3^q,x_4^q).
% \end{align*}
 Notice that $\mathrm{H}(3,q^2)$ and the Baer subspace are fixed under both $\tau$ and $\phi$. 
%more infor about these polarities
Recall from Section \ref{CRHs} that the construction of this family of relative hemisystems stemmed from stabilising a conic of the elliptic quadric $\mathrm{Q}^-(3,q)$ 
fixed by $\mathrm{P} \Omega^-(4,q)$ in $\mathrm{W}(3,q)$ \cite{MR3081646}. Clearly, $\tau$ and $\phi$ preserve the form defining $\mathcal{Q}^+$ given in Equation \ref{gammaquad}.
Let us consider the application of $\tau$ to the regulus $\mathscr{R}_1$.
\[
%   \left[ \begin{smallmatrix} \upsilon^q&1 &0&\lambda \\ \lambda \upsilon&\lambda &1&0 \\ \end{smallmatrix}\right] ^\tau & =  \left[ \begin{smallmatrix} \upsilon^q +1 & 1 &0&\lambda \\ \lambda\upsilon + \lambda&\lambda&1&0 \\ \end{smallmatrix}\right]\\
% & =  \left[ \begin{smallmatrix} \lambda\upsilon + \lambda&\lambda&1&0 \\ \upsilon^q +1 & 1 &0&\lambda \\ \end{smallmatrix}\right]\\
% & =  \left[ \begin{smallmatrix} \lambda(\upsilon^q +1) + \lambda&\lambda&1&0 \\ (\upsilon +1) +1 & 1 &0&\lambda \\ \end{smallmatrix}\right]\\
% & =  \left[ \begin{smallmatrix} \lambda\upsilon^q &\lambda&1&0 \\ \upsilon  & 1 &0&\lambda \\ \end{smallmatrix}\right]\\
% & =  \left[ \begin{smallmatrix} \upsilon^q &1&\eta&0 \\ \eta\upsilon  & \eta &0&1 \\ \end{smallmatrix}\right] \in \mathscr{R}_2
% \end{align*}
% where $\eta = \lambda^{-1}$. Moreover,
% \begin{align*}
 \left[ \begin{smallmatrix} 0&0&0&1 \\ \upsilon&1&0&0\\ \end{smallmatrix} \right]^\tau =  \left[ \begin{smallmatrix} 0&0&0&1 \\ \upsilon+1&1&0&0\\ \end{smallmatrix} \right] 
=  \left[ \begin{smallmatrix} 0&0&0&1 \\ \upsilon^q &1&0&0\\ \end{smallmatrix} \right] \in \mathscr{R}_2.
\]
Now, consider the application of $\phi$ to the reguli. For $\mathscr{R}_1$, we have the following:
% \begin{align*}
%  \left[ \begin{smallmatrix} \upsilon^q&1 &0&\lambda \\ \lambda \upsilon&\lambda &1&0 \\ \end{smallmatrix}\right] ^\phi & =  \left[ \begin{smallmatrix} (\upsilon^q)^q&1^q &0&\lambda^q \\ \lambda^q \upsilon^q&\lambda^q &1^q&0 \\ \end{smallmatrix}\right]\\
%   & =  \left[\begin{smallmatrix} \upsilon &1 &0&\lambda^q \\ \lambda^q \upsilon^q&\lambda^q &1&0 \\ \end{smallmatrix}\right]\\
%     & =  \left[\begin{smallmatrix}\lambda^q \upsilon^q&\lambda^q &1&0 \\  \upsilon &1 &0&\lambda^q \\ \end{smallmatrix}\right]\\
%     & = \left[ \begin{smallmatrix} \upsilon^q&1&\mu&0 \\ \mu \upsilon&\mu &0 &1 \\ \end{smallmatrix} \right] \in \mathscr{R}_2
% \end{align*}
% where $\mu = \lambda^{q^2-q-1}$.
% Additionally,
\begin{align*}
 \left[ \begin{smallmatrix} 0&0&0&1 \\ \upsilon&1&0&0\\ \end{smallmatrix} \right]^\phi &=  \left[ \begin{smallmatrix} 0&0&0&1 \\ \upsilon^q&1&0&0\\ \end{smallmatrix} \right] \in \mathscr{R}_2.
\end{align*}
Therefore, both $\phi$ and $\tau$ map $\mathscr{R}_1$ to $\mathscr{R}_2$. Since $\phi$ and $\tau$ have order 2, $\mathscr{R}_2$ must map to $\mathscr{R}_1$ under each of $\tau$ and $\phi$.
It follows that their product $\tau\phi$ fixes reguli.
Let $K = \langle \tau, \phi \rangle$, which is isomorphic to $\mathbb{Z}_2 \times \mathbb{Z}_2$.
Now, take the collineation group $J$ isomorphic to $\mathrm{PSL}(2,q)$ that fixes the hyperplane $\pi:x_2=0$. We may represent this as matrices of the form
\[
 \left( \begin{smallmatrix}
1&0&0&0\\
\sqrt{bf}+1&1&\sqrt{be}&\sqrt{cf}\\
\sqrt{bc}&0&b&c\\
\sqrt{ef}&0&e&f\end{smallmatrix} \right),
\]
where $b,c,e,f \in \mathrm{GF}(q)$ and $bf+ce=1$.
Define $\overline{G}= J \times K$ and $G = J \times \langle \tau\phi \rangle$.
We claim that these groups satisfy the conditions of Corollary \ref{Cosscoroll}.
 Firstly, notice that $G$ is contained in the intersection of $\overline{G}$ and $\PO(4,q)$ because $\mathrm{PSL}(2,q)$ is a subgroup of $\PO(4,q)$ \cite{wilson2009finite},
and $\tau\phi$ fixes the reguli of $\mathcal{Q}^+$, just as $\PO(4,q)$ does. Furthermore, if $g \in \overline{G} \cap \PO(4,q)$, then $g$ must fix the reguli of $\mathcal{Q}^+$, since $\PO(4,q)$ does. Therefore, $g \in G$ and we have shown that $G = \overline{G} \cap \PO(4,q)$. 
Furthermore, $\overline{G}$ is not contained in $\PO(4,q)$ because $\tau$ and $\phi$ do not fix the reguli of $\mathcal{Q}^+$. We must now show for all external points $P$ that $\overline{G}_P$ is not contained in $\PO(4,q)$.

% The line $\ell$ satisfying $x_2 = x_4 = 0$ meets all of the orbits of $\mathrm{PSL}(2,q)$ on points of $\pi$, and so we may take orbit representatives of the form $(1,0,X,0)$, where $X \in \mathrm{GF}(q^2)\setminus \mathrm{GF}(q)$. There are exactly $q$ elements of $\mathrm{PSL}(2,q)$ induced in the plane that fix the line $\ell$ because
% \begin{align*}
%  (1,0,X,0) \left( \begin{smallmatrix}
% 1 & 0 & 0 & 0 \\
% 0 & 1 & 0 & 0 \\
% \sqrt{bc} & 0 & b & c\\
% \sqrt{ef} & 0 & e & f \end{smallmatrix} \right) & = (1+X\sqrt{bc},0, bX, cX)
% \end{align*}
% which fixes the points on $\ell$ if and only if $c=0$ and $b=1$. Additionally, since $bf+ce = 1$, $f=1$. Therefore, the matrices in $\mathrm{PSL}(2,q)$ which fix $\ell$ pointwise are precisely those of the form
% \[
% \left( \begin{smallmatrix}
% 1 & 0 & 0 & 0 \\
% 0 & 1 & \sqrt{e} & 0 \\
% 0 & 0 & 1 & 0\\
% \sqrt{e} & 0 & e & 1 \end{smallmatrix} \right)
% \]
% where $e \in \mathrm{GF}(q)$. 
Let us first consider the external points that lie on the plane $\pi$. Consider the line $\ell$ satisfying $x_2=x_4=0$. A simple calculation shows that the stabiliser of $\ell$ in $J$ consists only of collineations represented by matrices of the form
\[
M_{b,e}:=
\begin{pmatrix}
1&0&0&0\\
0&1&\sqrt{be}&0\\
0&0&b&0\\
\sqrt{e/b}&0&e&1/b
\end{pmatrix}, \]
where $b,e\in \mathrm{GF}(q)$ and $b \neq 0$. Now if $\xi,\zeta \in\mathrm{GF}(q^2)\backslash\mathrm{GF}(q)$, then $(1,0,\xi,0)$ will be mapped to $(1,0,\zeta,0)$ by $M_{b,e}$ if and only if $b\xi=\zeta$. We can therefore define an equivalence relation on the points which lie on $\ell$ by $(1,0,\xi,0) \sim (1,0,\zeta,0) \Leftrightarrow \zeta = b\xi$, for some $b \in \mathrm{GF}(q) \setminus \{0\}$. In other words, two points on $\ell$ are related if they lie in the same orbit under $J$. Each equivalence class will have size $q-1$, and since the orbits of $J$ partition the external points of $\ell$, we find that $\ell$ must meet $q$ orbits of $J$ on external points. 
%Therefore,no two points of $\ell$ of the form $(1,0,\xi,0)$, where $\xi\notin \mathrm{GF}(q)$, lie in the same orbit under $\mathrm{PSL}(2,q)$. 

Note that the stabiliser of $(1,0,\xi,0)$ in $J$ is then $\{M_{1,e} \mid e \in \mathrm{GF}(q) \}$, which is a set of size $q$. By the Orbit-Stabiliser Theorem, the orbit of each point $(1,0,\xi,0)$, with $\xi\notin \mathrm{GF}(q)$, has size $q^2-1$. Since there are $(q^2-q)(q+1)$ external points in $\pi$, we see that the totally isotropic points $(1,0,\xi,0)$ satisfying $\xi\notin \mathrm{GF}(q^2)$ form a complete set of orbit representatives for $J$ acting on the external points in $\pi$.
Also notice that $\phi$ does not fix $\ell$, but $\tau$ does, and therefore $\tau$ lies in $\overline{G}_P$ for $P \in \ell$. Since $\tau$ switches reguli, it follows that $\overline{G}_P \nsubseteq \PO(4,q)$ for all $P \in \ell$, and since these points are orbit representatives for the action of $J$ on external points on $\pi$, $\overline{G}_P \nsubseteq \PO(4,q)$ for all external points $P$ on $\pi$.

%, and so $\tau\phi$ does not fix $\ell$. Therefore, the kernel of the action of $\overline{G}$ on $\ell$ has order $2q$, and the kernel of the action of $G$ on $\ell$ has order $q$. Furthermore, the stabiliser of any point $P= (1,0,\xi,0)$ under the action of $\overline{G}$ is the kernel of the action on $\ell$. Since $\tau$ is in the kernel of the action and it is contained in $\overline{G}$, but not $G$, we have $|\overline{G}_P : G_P| = 2$, and so $\overline{G}_P$ is not contained in $\PO(4,q)$ for all external points $P$ that lie on $\pi$.

Let $\mathcal{C}$ be the intersection of $\mathcal{Q}^+$ with $\pi$, which is a conic.
Now consider the external points which do not lie on $\pi$ or are collinear to a point on the conic $\mathcal{C}$. 
We take the line $n$ defined by the span of the point $Q = (1,0,1,0)$, which lies on $\pi$ and the point $R=(0,1,0,1)$. Notice that every point on $n$ except for $Q$ lies outside $\pi$. These points can be written in the form $(u,1,u,1)$, where $u \in \mathrm{GF}(q^2)\setminus \mathrm{GF}(q)$. 
In order to compute the stabiliser of each of these points in $J$, we consider

\begin{align*}
 (u,1,u,1) \left( \begin{smallmatrix}
1&0&0&0\\
\sqrt{bf}+1&1&\sqrt{be}&\sqrt{cf}\\
\sqrt{bc}&0&b&c\\
\sqrt{ef}&0&e&f\end{smallmatrix} \right) & = \left( (\sqrt{bc}+1)u + \sqrt{bf} +1 + \sqrt{ef}, 1, \sqrt{be} + e + bu, f + \sqrt{cf} + cu \right).
\end{align*}

Now, $f + \sqrt{cf} + cu =1$ and since $u \in \mathrm{GF}(q^2)\setminus \mathrm{GF}(q)$ and $c,f  \in \mathrm{GF}(q)$, we must have $c=0$. Therefore, $f=1$. Recall that we also have the relation $bf+ce=1$, which implies that $b=1$. Finally, 
$ u = u(1+ \sqrt{bc}) +1 + \sqrt{(b+e)f} = u + 1 + \sqrt{1+e}$ and therefore, $e=0$. Substituting these values of $b,c,e,f$ into the matrix, we have the identity matrix, and therefore the stabiliser of each of the points is trivial. 
Furthermore, by the Orbit-Stabiliser theorem, the orbit of each of these points must have size $q(q^2-1)$. Notice that since there are $q^2+1$ points on every line in $\mathrm{H}(3,q^2)$, we have
\[
 \sum_{S \in n \setminus\{ Q\}} |S^{J}| =  q(q^2-1)\times q^2.
\]
 Recall that there are a total of $q(q^2-1)(q^2+1)$ external points in $\mathrm{H}(3,q^2)$, and so only $q(q^2-1)$ external points are not covered by these orbits. These are the external points that are collinear with the conic $\mathcal{C}$. Therefore, the set of points on $n$ (excluding $Q$) form a transversal of the orbits of $J$ on external points which are not collinear with points of the conic $\mathcal{C}$, or lie on $\pi$. Now, consider the following element of $\overline{G}$:
\[
C = \left( \begin{smallmatrix}
1 & 0 & 0 & 0 \\
1 & 1 & 1 & 0 \\
0 & 0 & 1 & 0\\
1 & 0 & 1 & 1 \end{smallmatrix} \right)\\
 = \left( \begin{smallmatrix}
1 & 0 & 0 & 0 \\
0 & 1 & 1 & 0 \\
0 & 0 & 1 & 0\\
1 & 0 & 1 & 1\\
 \end{smallmatrix} \right)\tau\\
=D\tau. \\
\]
Notice that $D \in J$ and $C$ fixes all of the points $(u,1,u,1)$ on $n \setminus \{ Q \}$. But $C$ is not in $G$ because it swaps the reguli of $\mathcal{Q}^+$ (it is the product of an element of $J$ which fixes reguli, and $\tau$ which swaps them). Therefore, for all external points $P$ which do not lie on $\pi$ and are not collinear to any points on $\mathcal{C}$, $|\overline{G}_P : G_P| \neq 1$ and so $\overline{G}_P \nsubseteq \PO(4,q)$.

Finally, consider the $q(q^2-1)$ points which are collinear with some point of the conic $\mathcal{C}$. Let $m$ be the line spanned by $R = (1,0,0,0)$, which is the nucleus of the conic, and the point $(0,1,0,\gamma)$, where $\gamma \in \mathrm{GF}(q^2) \setminus \mathrm{GF}(q)$ such that $\gamma^q + \gamma = 1$. Define $W$ to be the set of $q$ external points incident with $\ell$. These can be written as $W= \{(0,1,0,\gamma) \} \cup \{ (1,v,0, \gamma v) : v\in \mathrm{GF}(q) \setminus \{ 0\} \}$.  We now calculate the stabiliser of a point in $W$:
\begin{align*}
 (1,v,0, \gamma v) \left( \begin{smallmatrix}
1&0&0&0\\
\sqrt{bf}+1&1&\sqrt{be}&\sqrt{cf}\\
\sqrt{bc}&0&b&c\\
\sqrt{ef}&0&e&f\end{smallmatrix} \right) & = \left(1+v(\sqrt{bf}+1) + \gamma v\sqrt{ef}, v, v\sqrt{be}+\gamma e v, v+\sqrt{cf} + \gamma fv\right).
\end{align*}
A simple calculation shows that in order for this to be equal to $ (1,v,0, \gamma v)$, we must have  $c=0, e=0, f=1, b=1$. Substituting these values in gives us the identity matrix, and so no orbit has fixed points under $J$. So the size of $W^{J} = q |J| = q^2(q^2-1)$.
Therefore, the points in $W$ form a transversal of the orbits of $J$ on external points that are collinear with a point of $\mathcal{C}$ . Consider the following element $A$ of $\overline{G}$:
\[
 A = \left( \begin{smallmatrix}
1 & 0 & 0 & 0 \\
1 & 1 & 0 & 1 \\
1 & 0 & 1 & 1\\
0 & 0 & 0 & 1 \end{smallmatrix} \right)\phi = B \phi.
\]
Clearly, $B\in J$, and so $A$ switches the reguli of $\mathcal{Q}^+$. Let us consider the action of $A$ on points of $W$.
\[
 (1,v,0, \gamma v) \left( \begin{smallmatrix}
1 & 0 & 0 & 0 \\
0 & 1 & 0 & 1 \\
1 & 0 & 1 & 1\\
0 & 0 & 0 & 1 \end{smallmatrix} \right)\phi 
 = (1, v, 0, v(\gamma^q +1) )\\
 = (1,v,0,\gamma v).
\]
Likewise,
\[
 (0,1,0, \gamma) \left( \begin{smallmatrix}
1 & 0 & 0 & 0 \\
0 & 1 & 0 & 1 \\
1 & 0 & 1 & 1\\
0 & 0 & 0 & 1 \end{smallmatrix} \right)\phi  = (0, 1, 0, 1 + \gamma)^\phi 
 = (0,1,0,\gamma).
\]
Therefore, $A$ fixes $W$ pointwise, and so is contained in $\overline{G}_P$ for all $P \in W$. But since $A$ switches the reguli of $\mathcal{Q}^+$, it cannot lie in $\PO(4,q)$, and so neither can $\overline{G}_P$ for all external points $P$ collinear with a point of the conic $\mathcal{C}$. Therefore, $G$ and $\overline{G}$ satisfy the conditions of Corollary \ref{Cosscoroll} and hence determine a relative hemisystem.

\end{proof}

\begin{remark}\label{remark:Cossidente}
We remark that although Cossidente describes one infinite family of relative hemisystems using his construction in \cite{MR3081646}, there are actually several more inequivalent infinite families of relative hemisystems that
arise from this construction. For $q=16$, we found by computer that there are five inequivalent relative hemisystems that admit $\mathrm{PSL}(2,q)$, and we conjecture that the number of inequivalent examples admitting this group increases with $q$.
\end{remark}

Cossidente's second family of relative hemisystems, admitting a group of order $q^2(q+1)$, for each $q$ even, described in \cite{cossidente2013new}, also satisfy the conditions of Corollary \ref{Cosscoroll}. To prove this, we first provide a concrete construction
of this family of relative hemisystems. As before, we consider the Hermitian space $\mathrm{H}(3,q^2)$ defined by the form $x_1x_2^{q} + x_2x_1^q +x_3x_4^q+x_4x_3^q$ over $\mathrm{GF}(q^2)$, with the 
embedded symplectic space $\mathrm{W}(3,q)$ defined as the restriction of the form to $\mathrm{GF}(q)$.
We explicitly define the following two hyperbolic quadrics.
\begin{equation}
\label{hypquad1}
 Q_1^+(3,q^2): \alpha x_1^2 + \beta x_2^2 + x_1x_2 + x_3x_4, 
 \end{equation}
 \begin{equation}
 \label{hypquad2}
 Q_2^+(3,q^2): \beta x_1^2 + \alpha x_2^2 + x_1x_2 + x_3x_4 
\end{equation}
where $\alpha \in \mathbb{F}_{q^2}$, $\beta = \alpha+1$ and $\alpha + \alpha^q + 1 = 0$. 
The reguli of $Q_1^+(3,q^2)$ are as follows:
\begin{align*}
\tilde{\mathscr{R}_1} & = \lbrace \left[ \begin{smallmatrix} \alpha&\alpha&0&\lambda\sqrt{\alpha} \\ \lambda\alpha&\lambda\beta&\sqrt{\alpha}&0 \\ \end{smallmatrix} \right] \mid \lambda  \in \mathrm{GF}(q^2) \rbrace \cup \lbrace \left[ \begin{smallmatrix}  0&0&0&1\\ \alpha & \beta & 0 & 0\\ \end{smallmatrix} \right]\rbrace\\
&\\
  \tilde{\mathscr{R}_2} & = \lbrace \left[ \begin{smallmatrix} \alpha&\alpha &\lambda\sqrt{\alpha}&0 \\ \lambda\alpha&\lambda \beta&0&\sqrt{\alpha} \\ \end{smallmatrix} \right] \mid \lambda  \in \mathrm{GF}(q^2) \rbrace \cup \lbrace \left[ \begin{smallmatrix}  0&0&1&0\\ \alpha & \beta & 0 & 0\\ \end{smallmatrix} \right]\rbrace.\\
  \end{align*}
  
\begin{theorem}
 Suppose $\overline{M}$ is the stabiliser in $\mathrm{PGU}(4,q)$ of the two hyperbolic quadrics $Q_1^+(3,q^2)$ and $Q_2^+(3,q^2)$ described above. Now, let $M$ be the stabiliser in $\overline{M}$ of a class
 of reguli in $Q_1^+(3,q^2)$.
Then $M$ is the group admitted by Cossidente's second family of relative hemisystems, and $M$ and $\overline{M}$ satisfy Corollary \ref{Cosscoroll}.
\end{theorem}

\begin{proof}
Firstly, we claim that  $\overline{M} = M :Z$, where $Z$ is the group generated by the involution $z$ defined by $(x_1, x_2, x_3, x_4) \mapsto (x_2^q, x_1^q, x_4^q, x_3^q)$. Firstly, notice that $|\overline{M} : M| = 2$, so $M$ is a normal subgroup of $\overline{M}$.
The action on lines of $\tilde{\mathscr{R}}_1$ (for instance) is permutation isomorphic to the action on a projective line $\mathrm{PG}(1,q^2)$. Now, $M$ fixes two
lines, say $\ell_1$ and $\ell_2$, which are in the intersection of $Q_1^+(3,q^2)$ and $Q_2^+(3,q^2)$, with $\ell_1 \in \tilde{\mathscr{R}}_1$ and $\ell_2 \in \tilde{\mathscr{R}}_2$.
The group $M$ fixing $\ell_1$ is permutation isomorphic to $\mathrm{AGL}(1,q^2)$ fixing a point on $\mathrm{PG}(1,q^2)$. Now, $\mathrm{AGL}(1,q^2)$ is transitive on the remaining points of $\mathrm{PG}(1,q^2)$ \cite[\textsection 7.7]{MR1409812},
and so $M$ is transitive on $\tilde{\mathscr{R}_1} \setminus \{ \ell_1\}$. Therefore, to show that $z$ maps $\tilde{\mathscr{R}}_1$ to $\tilde{\mathscr{R}}_2$, 
it is sufficient to prove it for $\ell_1$ and another line of $\tilde{\mathscr{R}}_1$. Let $\ell_{\infty , 1} = \left[ \begin{smallmatrix}  0&0&0&1\\ \alpha & \beta & 0 & 0\\ \end{smallmatrix} \right]$.
Now, 
\[
\ell_1^z  = \left[ \begin{smallmatrix} 1 & 1 &0&0 \\ 0&0&1&0\\ \end{smallmatrix} \right]^q \left( \begin{smallmatrix}
0 & 1 & 0 & 0 \\
1 & 0 & 0 & 0 \\
0 & 0 & 0 & 1\\
0 & 0 & 1 & 0 \end{smallmatrix} \right)
= \left[ \begin{smallmatrix} 1 & 1 &0&0 \\ 0&0&0&1\\ \end{smallmatrix} \right] \in \tilde{\mathscr{R}}_2,
\]
\[
 \ell_{\infty , 1}^z  = \left[ \begin{smallmatrix}  0&0&0&1\\ \alpha & \beta & 0 & 0\\ \end{smallmatrix} \right]^q \left( \begin{smallmatrix}
0 & 1 & 0 & 0 \\
1 & 0 & 0 & 0 \\
0 & 0 & 0 & 1\\
0 & 0 & 1 & 0 \end{smallmatrix} \right)
= \left[ \begin{smallmatrix}  0&0&1&0\\ \alpha & \beta & 0 & 0\\ \end{smallmatrix} \right] \in \tilde{\mathscr{R}}_2.
\]
A similar argument yields that the image of any line in $\tilde{\mathscr{R}}_2$ under $z$ is contained in $\tilde{\mathscr{R}}_1$.
Therefore, the involution $z$ switches the reguli of $Q_1^+(3,q^2)$ and so $M \cap Z = 1$. Since $z$ stabilises the two hyperbolic quadrics $Q_1^+$ and $Q_2^+$, we have $Z \leqslant M$, and therefore $\overline{M}= M :Z$.
 Let $D_1$ be the subgroup of $\mathrm{P} \Gamma \mathrm{U}(4,q)_{Q_1^+}$ that stabilises two families of reguli on $Q_1^+$. 
 Similarly, let $D_2$ be the subgroup of $\mathrm{P} \Gamma \mathrm{U}(4,q)_{Q_2^+}$ that stabilises two families of reguli on $Q_2^+$. 
We now prove that the orbits of $\overline{M}$ and $M$ on external points are identical.
Since $\overline{M} = M :Z$, it is sufficient to prove that for all $x \in \mathcal{P}_E, x^z \in x^M$.
Since $D_1$ is transitive on $\mathcal{P}_E$, this is equivalent to showing that
for all $g\in D_1$, we have $(P_0^g)^z \in (P_0^g)^M$ for all $P_0 \in \mathcal{P}_E$.

We claim that given a point $P_0 \in \mathcal{P}_E$, we have $P_0^z = P_0^m$ and $P_0 = (P_0^z)^z = (P_0^m)^m$ for some $m \in M$.
Since $D_1$ acts transitively on $\mathcal{P}_E$, it is sufficient for us to prove this claim for a specific point.
Let $P_0 = (1,0,1,0)$. Then we have the following
$$ P_0^z = (1,0,1,0)^q 
\left( \begin{smallmatrix}
0 & 1 & 0 & 0 \\
1 & 0 & 0 & 0 \\
0 & 0 & 0 & 1\\
0 & 0 & 1 & 0 \end{smallmatrix} \right) = (0,1,0,1).$$
Now, we must find an involution $m \in M$ such that $P_0^m = P_0^z = (0,1,0,1)$. This is equivalent to finding $m \in M$ such that $P_0^{zm^{-1}} = P_0^{zm}=P_0$. Take $m$ to be the following collineation:
\[
m = \phi \left( \begin{smallmatrix}
1 & 0 & 1 & 1 \\
0 & 1 & 1 & 1 \\
1 & 1 & 1 & 0\\
1 & 1 & 0 & 1 \end{smallmatrix} \right)
\]
where $\phi$ is the automorphism $x \mapsto x^q$. This collineation has the required property and so it only remains to show that $m \in M$. 

To show that $m \in M$, we must show that it is an involution and that it fixes the two hyperbolic quadrics that define $M$ and also each of the reguli in the intersection of the two hyperbolic quadrics.
Firstly recalling that we are working with a field with characteristic 2,
\[
 m^2 =  \left( \begin{smallmatrix}
1 & 0 & 1 & 1 \\
0 & 1 & 1 & 1 \\
1 & 1 & 1 & 0\\
1 & 1 & 0 & 1 \end{smallmatrix} \right) \left( \begin{smallmatrix}
1 & 0 & 1 & 1 \\
0 & 1 & 1 & 1 \\
1 & 1 & 1 & 0\\
1 & 1 & 0 & 1 \end{smallmatrix} \right) = I_4.
\]
Secondly, we must show that $m$ fixes the two hyperbolic quadrics $Q^+_1(3,q^2)$ and $Q_2^+(3,q^2)$.
Recall the defining quadratic forms of the two hyperbolic quadrics $Q_1^+(3,q^2)$ and $Q_2^+(3,q^2)$ given in Equations \ref{hypquad1} and \ref{hypquad2} respectively. 

The image of any point in $Q_1^+(3,q^2)$ under $m$ lies in $Q_1^+(3,q^2)$, so
$m$ fixes $Q^+_1(3,q^2)$, and by symmetry, it fixes $Q_2^+(3,q^2)$ as well.
Finally, we show that $m$ fixes the reguli of $Q_1^+(3,q^2)$. Again, we only need to test two lines from each regulus -- the line that is fixed by $M$ and another line.
For $\ell_1, \ell_{\infty , 1} \in \tilde{\mathscr{R}}_1$:
\[
\ell_1^m = \left[ \begin{smallmatrix} 1 & 1 &0&0 \\ 0&0&1&0\\ \end{smallmatrix} \right]^q
 \left( \begin{smallmatrix}
1 & 0 & 1 & 1 \\
0 & 1 & 1 & 1 \\
1 & 1 & 1 & 0\\
1 & 1 & 0 & 1 \end{smallmatrix} \right)
=  \left[ \begin{smallmatrix} 1 & 1 &0&0 \\ 1&1&1&0\\ \end{smallmatrix} \right] = \left[ \begin{smallmatrix} 1 & 1 &0&0 \\ 0&0&1&0\\ \end{smallmatrix} \right]\in \tilde{\mathscr{R}}_1,
\]
\[
\ell_{\infty , 1}^m  =  \left[ \begin{smallmatrix}  0&0&0&1\\ \alpha & \beta & 0 & 0\\ \end{smallmatrix} \right]^q  \left( \begin{smallmatrix}
1 & 0 & 1 & 1 \\
0 & 1 & 1 & 1 \\
1 & 1 & 1 & 0\\
1 & 1 & 0 & 1 \end{smallmatrix} \right) 
= \left[ \begin{smallmatrix}  1&1&0&1\\ \beta & \alpha & 0 & 0\\ \end{smallmatrix} \right]
=  \left[ \begin{smallmatrix}  \sqrt{\alpha}&\sqrt{\alpha}&0&\sqrt{\alpha}\\ \alpha & \beta & 1 & 0\\ \end{smallmatrix} \right] \in \tilde{\mathscr{R}}_1.
\]
Using a similar argument, $m$ fixes the reguli of $\tilde{\mathscr{R}}_2$ as well.
Therefore, $m$ preserves the reguli and so $m\in M$.

Now that we have proved the claim, we continue with the proof as before. We have $P_0^{gz} \in P_0^{gM}$ if and only if $P_0^{gzg^{-1}} \in P_0^{gMg^{-1}}$ for all $g \in D_1$. 
This is equivalent to $P_0$ having identical orbits under $\overline{M}^g$ and $M^g$ for all $g \in D_1$. By the Orbit-Stabiliser Theorem and since $|\overline{M} {:} M| = 2$, it follows that $| \overline{M}^g_{P_0}| = 2|M^g_{P_0}|$.
Finally, this holds if and only if $|\overline{M}_{P_0}| = 2|M_{P_0}|$, which is true by the claim proven above. Since $M_{P_0} = \overline{M}_{P_0} \cap \mathrm{P}\Omega^-(4,q)$, we have shown that $\overline{M}_{P_0} \nsubseteq \mathrm{P}\Omega^-(4,q)$ and 
since $D_1$ is transitive, this holds for all external points $P\in \mathcal{P}_E$.
Therefore, by Corollary \ref{Cosscoroll}, $M$ and $\overline{M}$ determine a relative hemisystem.
 
\end{proof}

Interestingly, we have found by computation in GAP \index{GAP}\cite{GAP4} that the relative hemisystem on $\mathrm{H}(3,q^2)$ arising from a Suzuki-Tits ovoid does not satisfy the criteria for Theorem \ref{majortheorem}.

We leave as an open question whether there are more relative hemisystems on $\mathrm{H}(3,q^2)$ fitting the criteria given in Theorem \ref{majortheorem}.

\section{A classification of the relative hemisystems on $\mathrm{H}(3,4^2)$}
Using a function written in GAP \cite{GAP4} and interfacing with Gurobi \cite{gurobi}, we were able to easily enumerate all of the relative hemisystems on the Hermitian space $\mathrm{H}(3,4^2)$. This result was previously unknown, or at least unmentioned in the literature.
\begin{proposition}
 There are 240 examples of relative hemisystems on the Hermitian space $\mathrm{H}(3,4^2)$, all of which are equivalent to the Penttila-Williford example on that Hermitian space.
\end{proposition}

Unfortunately, using the same approach does not allow us to classify all of the relative hemisystems on $\mathrm{H}(3, q^2)$ for $q \geqslant 8$ because the numbers of external points and external lines are significantly larger, making the computational problem much harder to solve. However, we were able to enumerate all of the relative hemisystems on $\mathrm{H}(3,8^2)$ with certain symmetry hypotheses.

We generated a computation tree using a branching technique outlined in \cite{Martis}, which essentially uses a partially ordered set of orbit representatives on $k$-tuples (closed under taking subsets) to decrease the search space.
This reduced the number of equivalent relative hemisystems found during computation, and made the search for relative hemisystems significantly more efficient.
\begin{proposition}
A relative hemisystem of $\mathrm{H}(3,8^2)$ is equivalent to one of the four known examples, or it has a trivial stabiliser.
\end{proposition}

\section*{Acknowledgements}
The authors would like to express their thanks to Prof. Gordon Royle for his assistance in computation for this paper, and Dr. Angela Aguglia for her insight into intersections of hyperbolic
quadrics with Hermitian spaces. The first author acknowledges the support of the Australian
Research Council Future Fellowship FT120100036. The second author acknowledges the support of a Hackett Postgraduate Research Scholarship. The third author acknowledges the support of the
Australian Research Council Discovery Grant DP120101336.

%\section*{References}
\nocite{fining}
\bibliographystyle{abbrv}
\bibliography{paperbib}
\vspace{1cm}
Contact details:
\end{document}